\documentclass[11pt,reqno]{amsart}
\usepackage{amsmath,amsfonts,amsthm,amscd,amssymb,graphicx}
\usepackage{epstopdf}
\numberwithin{equation}{section}
\usepackage{fullpage}
\usepackage{color}
\usepackage{enumitem}

\numberwithin{equation}{section}

\newtheorem{theorem}{Theorem}[section]

\newtheorem{lemma}{Lemma}[section]

\newtheorem{example}{Example}[section]

\theoremstyle{definition}
\newtheorem{definition}{Definition}[section]
\newtheorem{remark}{Remark}[section]

\newcommand{\RR}{{\mathbb R}}

\newcommand\bp{\begin{pmatrix}}
\newcommand\ep{\end{pmatrix}}
\newcommand\be{\begin{equation}}
\newcommand\ee{\end{equation}}

\renewcommand\a{\alpha}
\renewcommand\b{\beta}

\title{Traveling waves for  combustion reaction-diffusion-convection equations: \\ the full range of wave speeds}

\author{Pavel Dr\'{a}bek}
\address{Department of Mathematics and NTIS, Faculty of Applied Sciences, University of West Bohemia, Univerzitní 8, Plze\'{n} 30100, Czech Republic}
\email{pdrabek@kma.zcu.cz}

\author{Soyeun Jung}
\address{Division of International Studies, Kongju National University, Gongju-si, Chungcheongnamdo, South Korea}
\email{soyjung@kongju.ac.kr}

\author{Eunkyung Ko}
\address{Department of Mathematics, College of Natural Sciences, Keimyung University, Daegu  42601, South Korea}
\email{ekko@kmu.ac.kr}

\author{Michaela Zahradn\'{i}kov\'{a}}
\address{Department of Mathematics, Faculty of Science, University of South Bohemia,  Brani\v{s}ovsk\'{a}~1760, 370~05~\v{C}esk\'{e}~Bud\v{e}jovice, Czech Republic}
\email{zahram05@prf.jcu.cz}

\begin{document}

\begin{abstract} We consider traveling wave solutions to a reaction--diffusion--convection equation with a combustion-type reaction term. While a necessary condition for the existence of traveling waves is $c \geq H^*:=\sup_{0<u \leq \theta} \big(-\frac{1}{u} \int_0^u h(\sigma)\,d\sigma \big)$, where $c$ denotes the wave speed, $\theta\in(0,1)$ the ignition threshold, and $h$ the convective term, the available results in \cite{DZ25,MM03} establish existence and nonexistence only under the restriction $c \geq -\min_{u\in[0,1]} h(u)$. 
In this note, we close this gap by covering the entire range $c\ge H^*$.
\end{abstract}

\date{\today}

\maketitle







\section{Introduction}

We study traveling wave solutions to the reaction-diffusion-convection equation
\be
\label{rdc_com}
v_t=\left[d(v)|v_x|^{p-2}v_x\right]_x+h(v)v_x+g(v), \quad (x,t) \in \mathbb{R} \times [0,+\infty).
\ee
Here $p>1$, $h \in C[0,1]$ is a convective velocity, $d \in C^1(0,1)$ is a diffusion coefficient which may degenerate or become singular at one or both endpoints, and $g \in C[0,1]$ is a combustion-type reaction term, i.e., 
\be \label{r}
\text{$g(v)=0$ in $[0,\theta]$}, \quad  \text{$g(v)>0$ in $(\theta, 1)$}, \quad g(1)=0
\ee
for some $\theta \in (0,1)$. 

Reaction–diffusion–convection equations with combustion-type nonlinearities arise in models of flame propagation and interface dynamics, and traveling wave solutions play a fundamental role in their analysis. See \cite{DZ25, MM03} and the references therein for a detailed discussion of the model.

When seeking a traveling wave solution, we formally set $v(x,t)=u(x-ct)$, where $u$ is the wave profile and $c$ denotes the wave speed. Substituting this form, with $\xi=x-ct$, into \eqref{rdc_com} we obtain the profile equation
\be\label{pe}
 (d(u)|u'|^{p-2} u')' +(c+h(u) )u' +g(u) =0,
\ee
where $\prime=\frac{d}{d\xi}$. The objective of this note is to study monotone solutions of \eqref{pe} satisfying the boundary conditions
\be \label{bc}
\lim_{\xi \rightarrow -\infty} u(\xi) =1 \quad \mbox{and} \quad \lim_{\xi \rightarrow \infty} u(\xi) = 0.
\ee
Throughout the paper, we denote 
\be
h_m:=\min_{u \in [0,1]}h(u),  \quad  h_M:=\max_{u \in [0,1]}h(u), \quad  \quad H(u):=\int_0^u h(\sigma) d\sigma.
\notag
\ee

The existence and nonexistence of such solutions have been recently studied in \cite{DZ25}, extending and generalizing the results of \cite{MM03}, which were investigated for $p=2$, to the present setting. However, their results in \cite{DZ25, MM03} are obtained under the restriction 
\be \label{res}
c \geq -h_m,
\ee
where the effective speed $c+h(u)$ remains nonnegative for all $u \in[0,1]$, so that the convective effect never acts against the propagation of the traveling wave.

On the other hand, due to the combustion structure $g(u)=0$ in $[0,\theta]$, one obtains the necessary condition (see \eqref{nc1})
\be \label{nc}
c \geq H^*, \quad H^*:=\sup_{0<u \leq \theta} \Big(-\frac{H(u)}{u} \Big). 
\ee
Since $H^* \leq -h_m$, the restriction \eqref{res} imposed in \cite{DZ25, MM03} automatically implies \eqref{nc}. For this reason, the quantity $H^*$ does not explicitly appear in their analysis, and neither existence nor nonexistence results are available for wave speeds in the range
\be \label{regime}
H^* \leq c < -h_m.
\ee
The main purpose of the present note is to fill this gap. More precisely, we derive a sufficient condition for nonexistence for arbitrary wave speeds, as well as a sufficient condition ensuring that traveling waves may still exist even when the convective effect locally opposes the propagation of the wave.

As shown in \cite{DZ25}, the monotonicity of solutions to \eqref{pe}–\eqref{bc} is guaranteed. As a consequence, the problem can be reduced to a first-order boundary value problem (see \eqref{fode}), and all arguments can be carried out at this level. The restriction \eqref{res} imposed in \cite{DZ25, MM03} provides several technical advantages in the analysis of the associated first-order Cauchy problem, including uniqueness and monotonicity properties. In contrast, in our previous work on the monostable and bistable case (\cite{DJKZ24, DJKZ26}), we showed the first-order problem can still be analyzed even when the sign of $c+h(u)$ changes. By adapting this approach to the present combustion setting, we are able to remove the restriction \eqref{res}.

In recent years, traveling wave solutions for reaction–diffusion(–convection) equations with discontinuous density-dependent coefficients have received considerable attention (see \cite{DJKZ24,DJKZ26,GM26} and references therein). In the present combustion setting, however, we work under the same regularity assumptions as in \cite{DZ25}, since no substantial additional difficulties arise from lower regularity as long as the convection term $h$ is bounded and continuous at $0$. Accordingly, the sole focus of this paper is to fill the gap in the range of wave speeds. 

\section{Preliminaries}

In this section, we briefly recall several preliminaries from \cite{DJKZ24,DJKZ26,DZ25} that will be used throughout the paper, including the notion of solutions, the equivalent first-order boundary value problem, and the associated terminal value problem.

First, solutions of \eqref{pe}--\eqref{bc} are understood in the following sense.

\begin{definition}[\cite{DZ25}] \label{defsol} A continuous function $ u : \mathbb{R} \rightarrow [0,1]$ is a solution of  \eqref{pe}--\eqref{bc} if 
\begin{enumerate}[label=(\alph*)]
\item  $u \in C^1 (I_u)$, where $I_u := \{ \xi \in \mathbb{R}: 0 <u(\xi) <1  \}$, and \eqref{pe} holds at every $\xi \in I_u$;
\item the function $\xi \rightarrow d(u(\xi))|u'(\xi)|^{p-2} u'(\xi)$ is continuous on $\mathbb{R}$ 
and $d(u(\xi))|u'(\xi)|^{p-2} u'(\xi) \rightarrow 0$ as $u(\xi) \rightarrow 0$ and $u(\xi) \rightarrow 1$;
\item (boundary condition) $u(\xi) \rightarrow 1$ as $\xi \rightarrow -\infty$ and $u(\xi) \rightarrow 0$ as $\xi \rightarrow +\infty$. 
\end{enumerate}
\end{definition}
It was shown in \cite{DZ25} that every solution is strictly decreasing on the interval where $0<u<1$. Following \cite{DJKZ24, DJKZ26, DZ25}, we introduce
\be
y(u):=w(u)^{p'},  \quad   w(u):=-d(u)|u'|^{p-2}u'. 
\notag
\ee
Then the problem \eqref{pe}--\eqref{bc} reduces to the first-order boundary value problem
\be\label{fode}
\begin{cases}
y'(u) = p' \left[ (c+h(u)) (y^+ (u))^{1/p} -f(u)  \right],  \quad u \in (0,1), \\
 y(0)=0=y(1), 
\end{cases}
\ee
where  $y^+(u):=\max \{y(u),0\}$, $p'=p/(p-1)$, and $f(u) = d(u)^{p'-1}g(u)$.

According to \cite[Propostion 3.2]{DZ25}, solutions $u$ of \eqref{pe}--\eqref{bc} are in one-to-one correspondence with positive solutions $y \in C^1(0,1) \cap C[0,1]$ of the first-order problem \eqref{fode}. Therefore, throughout this note, we study the existence and nonexistence of traveling waves entirely through the analysis of positive solutions to \eqref{fode}.

Since $f \equiv 0$ on $[0, \theta]$, any positive solution $y$ of \eqref{fode} satisfies
\be
y'(u) = p' (c+h(u)) y(u)^{1/p}, \quad u \in (0,\theta), 
\notag
\ee
which is separable. Using $y(0)=0$, we obtain 
\be\label{fone1}
y(u)^{1/p'}=cu+H(u), \quad u \in [0, \theta].
\ee
Since $y>0$ on $(0,1)$, 
\be \label{nc1}
cu+H(u)>0 \quad \text{for all } u \in(0,\theta],
\ee
from which the necessary condition \eqref{nc} follows immediately. We also note that letting $u \to0+$ in \eqref{nc1} yields $c \geq -h(0)$, which is the necessary condition mentioned in \cite{DZ25, MM03}. While those works are restricted to the range \eqref{res}, our analysis focuses on wave speeds below $-h_m$. Consequently, the sharper condition \eqref{nc} plays a crucial role in the present work.

Lastly, recalling $f>0$ on $(\theta, 1)$, we consider the associated terminal value problem (TVP)
\be \label{tvp}
\begin{cases}
y'(u)
=
p'\left[(c+h(u))(y^+(u))^{1/p}-f(u)\right],
\quad u \in(0,1),\\
y(1)=0,
\end{cases}
\ee
and denote its solution by $\hat y_c$ for each $c \in \RR$. According to \cite{DJKZ24,DJKZ26}, $\hat y_c$ is positive on $(\theta,1)$, unique on every interval where it remains positive, and depends continuously on $c$. Moreover, if $c_1<c_2$, then
\be \label{mon}
\hat y_{c_1}(u)>\hat y_{c_2}(u)
\ee
whenever both solutions remain positive. In addition, under the assumption
\be \label{mu_com}
\mu:=\sup_{u \in (\theta, 1)} \frac{f(u)}{(u-\theta)^{p'-1}} <\infty,
\ee
it was shown in \cite{DJKZ24, DJKZ26} that there exists a threshold wave speed
\be \label{cBcom}
c_B\in\big[-h(\theta),\, -h_m+(p')^{1/p'}p^{1/p}\mu^{1/p'}\big]
\ee
such that the problem \eqref{tvp} on $[\theta,1]$, together with the condition $y(\theta)=0$, admits a unique positive solution if and only if 
\be \label{cB}
c \geq c_B.
\ee
The existence of this threshold will play a crucial role in proving the nonexistence of positive solutions to \eqref{fode}. Indeed, if $c \geq c_B$, then $\hat y_c(\theta)=0$, and hence no positive solution of \eqref{fode} can exist on $(0,1)$.

\section{Nonexistence result}

We begin with the nonexistence of traveling wave solutions to \eqref{pe}--\eqref{bc} under the necessary condition \eqref{nc}, which in turn implies nonexistence for all wave speeds.

\begin{theorem}[Nonexistence] \label{Nonex}
Let \eqref{mu_com} hold and assume that
\be \label{nonex3}
H(\theta)  
\geq  -H^*+(1-\theta)(-h_m+(p')^{1/p'}p^{1/p}\mu^{1/p'}).
\ee
Then the problem \eqref{fode} has no positive solution for any $c\geq H^*$; in particular, it admits no positive solution for any $c \in \RR$. 
\end{theorem}

\begin{proof}
We argue by contradiction. Suppose that \eqref{fode} admits a positive solution $y_c$ for some $c\geq H^*$. Since $y_c$ also solves the TVP \eqref{tvp} on $[\theta,1]$, the threshold property \eqref{cB} implies that necessarily $c<c_B$. Divide \eqref{fode} by $y_c(t)^{1/p}$ and integrate over $(\theta,1)$ to obtain
\be \label{non1}
y_c(\theta)^{1/p'} =-c(1-\theta)-H(1)+H(\theta) +\int_{\theta}^1 \frac{f(u)}{y_c(u)^{1/p}}\,du.
\ee
Applying the same identity to $\hat y_{c_B}$ and using $\hat y_{c_B}(\theta)=0$, we have 
\be \label{non2}
\int_{\theta}^1 \frac{f(u)}{\hat y_{c_B}(u)^{1/p}}\,du = c_B(1-\theta)+H(1)-H(\theta).
\ee
Since $c<c_B$ and $f>0$ on $(\theta, 1)$, the monotonicity \eqref{mon} yields
\be
\int_{\theta}^1 \frac{f(u)}{y_c(u)^{1/p}}\,du = \int_{\theta}^1 \frac{f(u)}{\hat y_c(u)^{1/p}}\,du < \int_{\theta}^1 \frac{f(u)}{\hat y_{c_B}(u)^{1/p}}\,du. 
\notag
\ee
Using \eqref{cBcom}, $c \geq H^*$, and the assumption \eqref{nonex3},  it follows from \eqref{non1}--\eqref{non2} that 
\be \label{upper_y}
y_c(\theta)^{1/p'} < (1-\theta)(c_B-c) \leq (1-\theta)\big(-h_m+(p')^{1/p'}p^{1/p}\mu^{1/p'}-H^*\big) \leq H^*\theta +H(\theta). 
\ee

On the other hand, again using $c \geq H^*$, we deduce from \eqref{fone1} that 
\be 
y_c(\theta)^{1/p'} \geq H^*\theta+H(\theta),
\notag
\ee
which contradicts \eqref{upper_y}. This completes the proof. 
\end{proof}

\begin{remark}\label{recomnon}
In \cite[Theorem~2.4]{DZ25}, nonexistence is proved for $c \geq -h_m$  under the assumptions
\be \label{assDZnon}
\int_0^1 f(u) \, du < \infty \quad \text{and} \quad H(\theta) > \theta h_m + \Big(p'\int_0^1 f(u) \, du \Big)^{1/p'}. 
\ee
Since our result yields nonexistence for any real value of $c$, the assumptions \eqref{mu_com} and \eqref{nonex3} imply \eqref{assDZnon}. Indeed, by \eqref{mu_com},
\be
\int_0^1 f(u)\,du = \int_{\theta}^1 \frac{f(u)}{(u-\theta)^{p'-1}}(u-\theta)^{p'-1}\,du \leq \mu \int_{\theta}^1 (u-\theta)^{p'-1}\,du = \frac{\mu}{p'}(1-\theta)^{p'} < \infty. 
\notag
\ee
Since $p>1$ and $H^* \leq -h_m$, it follows from \eqref{nonex3} that
\be
\theta h_m+\Big(p'\int_0^1 f(u)\,du\Big)^{1/p'}
< \theta h_m+(1-\theta)(p')^{1/p'}p^{1/p}\mu^{1/p'}-H^*-h_m 
\le H(\theta).
\notag
\ee
\end{remark}

\begin{example}
Theorem \ref{Nonex} is particularly relevant when $H^*<-h_m$, since otherwise the nonexistence result in \cite{DZ25} already applies. For example, let $p=p'=2$, $\theta=\frac12$, $d(u) =1$ on $[0,1]$,  
\be
g(u)=
\begin{cases}
0, & 0\le u \le \frac12,\\
\frac1{16}\left(u-\frac12\right)(1-u), & \frac12<u \le 1,
\end{cases}
\qquad
\text{and}
\qquad 
h(u)=
\begin{cases}
8u, & 0\le u\le \frac12,\\
9-10u, & \frac12<u \le 1.
\end{cases}
\notag
\ee
Then $H^*=0 < 1=-h_m$, and \eqref{nonex3} holds. Thus no traveling waves exist for any wave speed.
\end{example}

\section{Existence result}

Recall that the solution $\hat y_c$ of \eqref{tvp} is positive on $(\theta,1)$.  For each $c \geq H^*$, let $(u_c,1)\subset(0,1)$ denote the maximal interval such that
\be
\hat y_c(u)>0 \quad  \text{on $(u_c,1)$}. 
\notag
\ee
In what follows, we regard $\hat y_c$ as a positive solution of \eqref{tvp} restricted to the interval $(u_c,1)$:
\be \label{tvp_ex}
\begin{cases}
y'(u)
=
p'\left[(c+h(u))y(u)^{1/p}-f(u)\right],
\quad u \in(u_c,1),\\
y(1)=0. 
\end{cases}
\ee

We now state the existence of positive solutions to \eqref{fode}, recalling $\kappa(p)$ introduced in \cite{DZ25}:
\begin{align}\label{kappa}
\kappa(p)
&=
\left\{
\begin{array}{ll}
1/(2^{p'-1}-1), & 1<p<2, \\[1ex]
1, & p=2, \\[1ex]
p'/\big(p'-1+\hat{\kappa}(p')\big), & p>2,
\end{array}
\right.
\qquad 
\hat{\kappa}(r)
=
\dfrac{1+r(r-1)^{\frac{1}{r-2}}+(r-1)^{\frac{r}{r-2}}}
{\big(1+(r-1)^{\frac{1}{r-2}}\big)^r}.
\notag
\end{align}

\smallskip

\begin{theorem}[Existence] \label{ex_tvp}
Let $f \in L^1(0,1)$ and assume that
\be \label{concomex_tvp}
H^*+h_M \leq \Big(\kappa(p) \int_0^1 f(u)\, du \Big)^{1/p'}.
\ee
Then there exists a unique $c^* > H^*$ such that the problem \eqref{fode} has a unique  positive solution. Moreover, $c^*$ satisfies
\be \label{c^*}
c^* \leq \max \Big\{
\frac{1}{\theta}
\Big[
\Big(
p'\int_0^1 f(u) \, du
\Big)^{1/p'}
- H(\theta)
\Big], ~ -h_m \Big\}.
\ee
\end{theorem}

\begin{remark} Our existence result is obtained through a TVP formulation as in \cite{DZ25}. However, if one adopts a forward initial value approach as in \cite{MM03}, then one may expect existence for $c^* \geq H^*$ under the corresponding strict inequality assumption in \eqref{concomex_tvp}.
\end{remark}


\smallskip

Before proving Theorem \ref{ex_tvp}, we first establish the following lemma, which is a modification of \cite[Lemma B.2]{DZ25} adapted to our setting.

\begin{lemma} \label{lem} If $f \in L^1(0,1)$ and \eqref{concomex_tvp} holds, then 
\be
\hat y_{H^*}(\theta)^{1/p'}>H^*\theta+H(\theta).
\ee
\end{lemma}
\begin{proof}
To argue by contradiction, assume that 
\be \label{lemmaB.2_1}
\hat y_{H^*}(\theta)^{1/p'} \leq H^*\theta+H(\theta).
\ee
For a given $\tau \in (\theta, 1)$, dividing \eqref{tvp_ex} by $y_c(u)^{1/p}$ and integrating it with $c=H^*$ over $(\theta, \tau)$ gives 
\be \label{comtau_tvp}
\begin{split}
\hat y_{H^*}(\tau)^{1/p'}
& =\hat y_{H^*}(\theta)^{1/p'} +\int_{\theta}^{\tau} (H^*+h(u)) \,du-\int_{\theta}^{\tau} \frac{f(u)}{\hat y_{H^*}(u)^{1/p}} \, du \\
& < \hat y_{H^*}(\theta)^{1/p'} + (H^*+h_M)(1-\theta). 
\end{split}
\ee

On the other hand, integrating \eqref{tvp_ex} over $(\theta, 1)$ with $c=H^*$ and applying the mean value theorem, there exists $\tau^* \in (\theta, 1)$ such that 
\be
\begin{split}
0  
& \leq  \hat y_{H^*}(\theta)+ p' (H^*+h_M) \int_{\theta}^1\hat y_{H^*}(u)^{1/p}\,du-p'\int_{\theta}^1 f(u) \, du \\
& = \hat y_{H^*}(\theta)+p'\hat y_{H^*}(\tau^*)^{1/p}(H^*+h_M)(1-\theta)-p'\int_{\theta}^1 f(u)\, du. 
\notag
\end{split}
\ee  
It follows from \eqref{lemmaB.2_1} and \eqref{comtau_tvp} that 
\be \label{w2_tvp}
0  < (H^*\theta+H(\theta))^{p'} +p' \big(H^*\theta+H(\theta) + (H^*+h_M)(1-\theta) \big)^{p'-1} (H^*+h_M)(1-\theta) -p'\int_{0}^1 f(u) \,du.  
\notag
\ee
Here, we observe that both $\a:=H^*\theta+H(\theta)$ and $\b:=(H^*+h_M)(1-\theta)$ are nonnegative. The above inequality has exactly the same form as (B.7) in \cite{DZ25}, with $H(\theta)$ and $H(1)-H(\theta)$ replaced by $\a$ and $\b$, respectively. Therefore, arguing exactly as below (B.7), one obtains a contradiction to \eqref{concomex_tvp}. The key difference from \cite{DZ25} is that
\be 
0 \leq \a+\b=\int_0^{\theta} (H^*+h(u))\,du+\int_{\theta}^1 (H^*+h_M)\,du \leq H^*+h_M,
\notag
\ee
which yields the left--hand side of \eqref{concomex_tvp}. In contrast, the quantity $H(1)=H(\theta)+(H(1)-H(\theta))$ appears in \cite{DZ25}. 
\end{proof}

\smallskip

We are now ready to prove Theorem \ref{ex_tvp}.

\begin{proof}[The proof of Theorem  \ref{ex_tvp}] 

Following the strategy of \cite{MM03}, but using a bacward terminal value approach instead of the forward initial value formulation, we define two sets
\be
\mathcal A:=\{ c  \geq H^* : u_c =0, ~ \hat y_c(0)>0\}, \quad 
\mathcal B = \{ c \geq H^*: u_c >0, ~\hat y_c(u_c)=0 \}.
\notag
\ee
Thanks to the continuity of $\hat y_c$ and the continuous dependence of $\hat y_c$ on $c$, both $\mathcal A$ and $\mathcal B$ are open. Our strategy is to show that $\sup \mathcal A=\inf \mathcal B$. We also note that if $c \notin \mathcal A\cup\mathcal B$, then $\hat y_c$ is a positive solution of \eqref{fode}. 

We first prove that both $\mathcal A$ and $\mathcal B$ are nonempty. Assume that $\mathcal B=\emptyset$. Then, for every $c\geq H^*$, we have $u_c=0$ and $\hat y_c(0)\geq0$. 
Dividing \eqref{tvp_ex} by $y_c(u)^{1/p}$ and integrating over $[0,\theta]$ yields
\be
\hat y_c(\theta)=\big(\hat y_c(0)^{1/p'}+c\theta+H(\theta)\big)^{p'}\geq (c\theta+H(\theta))^{p'}.
\notag
\ee
On the other hand, for $c>-h_m$, integrating \eqref{tvp_ex} over $[\theta,1]$ gives
\be
\hat y_c(\theta) = -p'\int_{\theta}^1 (c+h(u)) \hat y_c(u)^{1/p}\,du + p'\int_{\theta}^1 f(u)\,du < p'\int_{\theta}^1 f(u)\,du.
\notag
\ee
Consequently,
\be
(c\theta+H(\theta))^{p'} < p'\int_{\theta}^{1} f(u)\,du,
\notag
\ee
which is impossible for sufficiently large $c$ satisfying
\be \label{clarge}
c > -h_m \quad \text{and} \quad 
c \geq  \frac{1}{\theta} \left[ \left(p'\int_{\theta}^{1} f(u)\,du\right)^{1/p'} -H(\theta) \right].
\ee
This contradiction proves that $\mathcal B\neq\emptyset$.

We now claim $H^* \in \mathcal A$ by proving $\hat y_{H^*}(u_{H^*})>0$. Dividing \eqref{tvp_ex} by $y_c(u)^{1/p}$ and integrating over $[u_{H^*},\theta]$ gives
\be
\hat y_{H^*}(\theta)^{1/p'}-\hat y_{H^*}(u_{H^*})^{1/p'}=H^*(\theta-u_{H^*})+H(\theta)-H(u_{H^*}). 
\ee
Hence, by Lemma \ref{lem} and the fact that $u_{H^*}\in[0,\theta]$,
\be
\hat y_{H^*}(u_{H^*})^{1/p'} > H^* u_{H^*}+H(u_{H^*}) \geq 0,
\ee
which implies that $u_{H^*}=0$ and $H^* \in \mathcal A$. In particular, $\mathcal A\neq\emptyset$.

Since $\hat y_c$ is strictly decreasing with respect to $c$, it follows that $\sup \mathcal A\leq \inf \mathcal B$. Suppose that $\sup \mathcal A<\inf \mathcal B$. Then there exist $c_1<c_2$ such that $c_1,c_2\notin\mathcal A\cup\mathcal B$. Hence, both $\hat y_{c_1}$ and $\hat y_{c_2}$ are positive solutions of \eqref{fode}, and thus $\hat y_{c_2}(u)<\hat y_{c_1}(u)$ on $(0,1)$. However, this contradicts \eqref{fone1}, since
\be
\hat y_{c_1}(u)=(c_1u+H(u))^{p'}<(c_2u+H(u))^{p'}=\hat y_{c_2}(u) \quad \text{on $(0,\theta)$}. 
\ee
Therefore, $\sup \mathcal A=\inf \mathcal B$. Since $\mathcal A$ and $\mathcal B$ are open, the critical value $c^*:=\sup \mathcal A=\inf \mathcal B$ satisfies $c^*\notin\mathcal A\cup\mathcal B$. Hence, \eqref{fode} has a unique positive solution $y_c$ if and only if $c=c^*$.

If $c^*$ is sufficiently large so as to satisfy \eqref{clarge}, then necessarily $c^*\in\mathcal B$ by the same argument used to prove that $\mathcal B\neq\emptyset$. Since $c^*\notin\mathcal B$, the estimate \eqref{c^*} follows.
\end{proof}

\begin{example}
The novelty of our existence result lies in showing the possibility of traveling waves with wave speeds in the range \eqref{regime}. More precisely, the assumption \eqref{concomex_tvp} suggests that if the overall adverse convective effect remains sufficiently weaker than the propagation mechanism generated by the reaction--diffusion structure, then traveling waves can still persist even in regions where the convective effect acts against the propagation direction. Such a situation indeed occurs when $f$ and $h$ simultaneously satisfy the existence condition \eqref{concomex_tvp} and the nonexistence condition \eqref{assDZnon} from \cite{DZ25}, the latter guaranteeing nonexistence for all $c\ge -h_m$. 

For example, this may occur when $H^*+h_M$ is relatively small, while $H(\theta)-\theta h_m$ is relatively large. Let $p=p'=2$, $\theta=\frac12$, $d(u) \equiv 1$, $h(u)=-1-u$ on $[0,1]$, and  
\be
g(u)=
\begin{cases}
0, & 0\le u\le \frac12,\\[1mm]
\frac{78}{25}\left(u-\frac12\right)(1-u), & \frac12<u<1.
\end{cases}
\notag
\ee
Then $h_m=-2$, $h_M=-1$, and $H^*=\frac{5}{4}<-h_m$. A direct computation shows that both \eqref{concomex_tvp} and \eqref{assDZnon} are satisfied. Consequently, the corresponding wave speed $c^*$ necessarily satisfies \eqref{regime}.
\end{example}

\end{document}